\documentclass[12pt,twoside,reqno]{amsart}
\usepackage{amsxtra,amscd}
\usepackage{graphicx}
\usepackage{amsmath}
\usepackage{amsfonts}
\usepackage{amssymb}

\newtheorem{theorem}{Theorem}[section]

\theoremstyle{definition}
\newtheorem{definition}[theorem]{Definition}

\newtheorem{remark}[theorem]{Remark}

\newtheorem{corollary}[theorem]{Corollary}

\newtheorem{condition}[theorem]{Condition}

\numberwithin{equation}{section}

\begin{document}
\title{Notes on the quasi-galois closed schemes}
\author{Feng-Wen An}
\address{School of Mathematics and Statistics, Wuhan University, Wuhan,
Hubei 430072, People's Republic of China} \email{fwan@amss.ac.cn}
\subjclass[2000]{Primary 14H30}
\keywords{affine covering, quasi-galois}

\begin{abstract}
Let $f:X\rightarrow Y$ be a surjective morphism of integral schemes. Then $X$ is said to be quasi-galois closed over $Y$ by $f$ if $X$ has a unique conjugate over
$Y$ in an algebraically closed field. Such a notion has been applied to the computation of \'{e}tale fundamental groups. In this paper we will use affine coverings with values in a fixed field to discuss quasi-galois closed and then give
a sufficient and essential condition for quasi-galois closed. Here, we will avoid using affine structures on a scheme since their definition looks copious and fussy.
\end{abstract}

\maketitle





{\tiny{
\begin{center}
{Contents}
\end{center}

\qquad {Introduction}

\qquad {1. Preliminaries}

\qquad {2. Definition for \emph{qc} schemes}

\qquad {3. A condition equivalent to \emph{qc} schemes}

\qquad {References}}}

\section*{Introduction}

The quasi-galois closed schemes are introduced and then discussed by affine structures in \cite{An2}. That is, let $X/Y$ be integral schemes with a surjective morphism $f$. Then $X$ is said to be quasi-galois closed (or \emph{qc} for short) over  $Y$ by $f$ if $X$ has a unique conjugate over $Y$ in an algebraically closed field. The \emph{qc}  schemes behave like
quasi-galois extensions of fields and have many nice properties, which can be regarded as a generalization of pseudo-galois schemes in the sense of Suslin-Voevodsky (see \cite{VS1,SV2}).

The notion of \emph{qc} schemes introduced there is mainly for us to understand the \'{e}tale fundamental group of a scheme. In deed, by \emph{qc} schemes we have the computation of \'{e}tale fundamental groups of schemes (see \cite{An4,An6}) and get a splitting homotopy exact sequence of these profinite groups in the sense of Grothendieck (see \cite{An4}). By \emph{qc} covers of a scheme, we will also obtain a profinite group, a \emph{qc} fundamental group of the scheme that contains the \'{e}tale fundamental group as a normal subgroup (see \cite{An5,An6}).

However, the definition for affine structures on a scheme looks rather copious and fussy. In this paper we will avoid using affine structures to get the key property for \emph{qc} schemes. Instead, we will use affine coverings  with values in an algebraically closed field to discuss \emph{qc} schemes and then obtain
a sufficient and essential condition for a \emph{qc} scheme.

\subsection*{Acknowledgements}

The author would like to express his sincere gratitude to Professor
Li Banghe for his advice and instructions on algebraic
geometry and topology.

\section{Preliminaries}

\subsection{Conventions}

As usual, for an  integral scheme $X$, we
let $k(X)$ denote the function field of $X$.

For an integral domain $D$, let $Fr(D)$ denote the field of
fractions of $ D $. In particular, $Fr(D)$ will be assumed to be contained in $\Omega$ if $D$ is contained in a field
$\Omega$.

\subsection{Affine covering with values}

Let $X$ be a scheme. An
\textbf{affine covering} of $X$ is a
family $\mathcal{C}_{X}=\{(U_{\alpha },\phi _{\alpha };A_{\alpha
})\}_{\alpha \in \Delta }$, where for each $\alpha \in \Delta $,
 $\phi _{\alpha }$ is an isomorphism from an open set $U_{\alpha
}$ of $X$ onto the spectrum $Spec{A_{\alpha }}$ of a
commutative ring $A_{\alpha }$.

Each $(U_{\alpha
},\phi _{\alpha };A_{\alpha })\in \mathcal{C}_{X}$ is called a
\textbf{local chart}. For the sake of brevity, a local chart
$(U_{\alpha},\phi_{\alpha};A_{\alpha })$  will be denoted
by $U_{\alpha}$ or $(U_{\alpha},\phi_{\alpha})$.

An affine covering
$\mathcal{C}_{X}$ of
$(X, \mathcal{O}_{X})$ is said to be \textbf{reduced} if
$U_{\alpha}\neq U_{\beta} $ holds for any $\alpha\neq \beta$ in
$\Delta$.

Let $\mathfrak{Comm}$ be the category of commutative rings with
identity. For a given field $\Omega$, let
$\mathfrak{Comm}(\Omega)$ be the category consisting of the subrings
of $\Omega$ and their isomorphisms.

\begin{definition}
Let $\mathfrak{Comm}_{0}$ be a subcategory of
$\mathfrak{Comm}$. An affine covering
$\{(U_{\alpha},\phi_{\alpha};A_{\alpha })\}_{\alpha \in \Delta}$ of
$X$ is said to be \textbf{with values} in
$\mathfrak{Comm}_{0}$ if
 for each $\alpha \in \Delta $ there are $\mathcal{O}_{X}(U_{\alpha})=A_{\alpha}$ and $U_{\alpha}=Spec(A_{\alpha})$, where
 $A_{\alpha }$ is a ring contained in $\mathfrak{Comm}_{0}$.

In particular, let $\Omega$ be a field. An affine covering
$\mathcal{C}_{X}$ of $X$ with values in
$\mathfrak{Comm}(\Omega)$ is said to be \textbf{with values in the
field $\Omega$}.
\end{definition}

\begin{remark}
The affine open subschemes of a given scheme are usually flexible and unspecified. Here, the field $\Omega$ above enables the affine open subschemes of $X$ to be \emph{measurable} while the other open subschemes of $X$ are still \emph{unmeasurable}.
\end{remark}

\section{Definition for \emph{qc} Schemes}

 Assume that $\mathcal{O}_{X}$ and $\mathcal{O}^{\prime}_{X}$ are two structure sheaves on the underlying space of an integral scheme $X$. The two integral schemes $(X,\mathcal{O}_{X})$ and $(X, \mathcal{O}^{\prime}_{X})$ are said to be \textbf{essentially equal} provided that for any open set $U$ in $X$, we have
 $$U \text{ is affine open in }(X,\mathcal{O}_{X}) \Longleftrightarrow \text{ so is }U \text{ in }(X,\mathcal{O}^{\prime}_{X})$$ and in such a case,  $D_{1}=D_{2}$ holds or  there is $Fr(D_{1})=Fr(D_{2})$ such that for any nonzero $x\in Fr(D_{1})$, either $$x\in D_{1}\bigcap D_{2}$$ or $$x\in D_{1}\setminus D_{2} \Longleftrightarrow x^{-1}\in D_{2}\setminus D_{1}$$ holds, where $D_{1}=\mathcal{O}_{X} (U)$ and $D_{2}=\mathcal{O}^{\prime}_{X} (U)$.

 Two schemes $(X,\mathcal{O}_{X})$ and $(Z,\mathcal{O}_{Z})$ are said to be \textbf{essentially equal} if the underlying spaces of $X$ and $Z$ are equal and the schemes $(X,\mathcal{O}_{X})$ and $(X,\mathcal{O}_{Z})$ are essentially equal.

Let $X$ and $Y$ be integral schemes and let $f:X\rightarrow Y$ be a
surjective morphism of finite type. Denote by
$Aut\left( X/Y\right)$ the automorphism group of $X$ over
$Y$.

A integral scheme $Z$ is
said to be a \textbf{conjugate} of $X$ over $Y$ if there is an
isomorphism $\sigma :X\rightarrow Z$ over $Y$.

\begin{definition}
$X$ is said to be \textbf{quasi-galois closed} (or \textbf{\emph{qc}} for short)  over $Y$
by $f$ if  there is an algebraically closed field $\Omega$
and a reduced affine covering $\mathcal{C}_{X}$ of $X$ with values in $
\Omega $ such that for any conjugate $Z$ of
$X$ over $Y$ the two conditions are satisfied:
\begin{itemize}
\item $(X,\mathcal{O}_{X})$ and $(Z,\mathcal{O}_{Z})$ are essentially equal if $Z$ has a reduced
affine covering with values in $\Omega$.

\item $\mathcal{C}_{Z}\subseteq \mathcal{C}_{X}$ holds if $\mathcal{C}_{Z}$
is a reduced affine covering of $Z$ with values in $\Omega $.
\end{itemize}
\end{definition}

There are many \emph{qc} schemes. For example, let $t$ be a variable over $\mathbb{Q}$. Then $Spec(\mathbb{Z}[2^{1/2},t^{1/3}])$ is \emph{qc} over $Spec(\mathbb{Z})$, where we take $\Omega=\overline{\mathbb{Q}(2^{1/2},t^{1/3})}$.

\begin{remark} Now let us give some notes on the algebraically closed field $\Omega$ in \emph{Definition 2.1}.

$(i)$ By $\Omega$, the rings of affine open sets in $X$ are taken to be as subrings of the same ring $\Omega$ so that they can be compared with each other.

$(ii)$ By $\Omega$, we can restrict ourselves only to consider the function fields which have the same variables over a given field.
\end{remark}

\begin{remark}
The affine covering $\mathcal{C}_{X}$ of $X$ in \emph{Definition 2.1} is maximal by set inclusion, which is indeed the \emph{natural affine structure} of $X$ with values in $\Omega$. Conversely, it can be seen that a \emph{qc} integral scheme has a unique natural affine structure with values in $\Omega$ (see \cite{An2}).
\end{remark}

\begin{remark}
It is seen that \emph{qc} schemes behave always like quasi-galois extensions of fields (see \cite{An2}). In general, there are an infinite number of \emph{qc} schemes over a given integral scheme. See \cite{An3,An4,An5} for the criterion and existence of \emph{qc} schemes.
\end{remark}

\section{A Condition Equivalent to \emph{qc} Schemes}

Let $K$ be an extension of a field $k$. Note that here $K$ is not necessarily algebraic over $k$.

\begin{definition}
$K$ is said to be \textbf{quasi-galois} over $k$ if each
irreducible polynomial $f(X)\in F[X]$ that has a root in $K$ factors
completely in $K\left[ X\right] $ into linear factors for any subfield $F$ with $k\subseteq F\subseteq K$.
\end{definition}

Let $D\subseteq D_{1}\cap D_{2}$ be three integral domains. Then $D_{1}$
is said to be \textbf{quasi-galois} over $D$ if $Fr\left(
D_{1}\right) $ is quasi-galois over $Fr\left( D\right) $.

\begin{definition}
$D_{1}$ is said to be a \textbf{conjugation} of $D_{2}$ over $D$ if
there is an $F-$isomorphism $\tau:Fr(D_{1})\rightarrow
Fr(D_{2})$ such that
$
\tau(D_{1})=D_{2}$,
where $F\triangleq k(\Delta)$, $k\triangleq Fr(D)$, $\Delta$ is a transcendental basis of the field $Fr(D_{1})$ over $k$, and $F$ is contained in $Fr(D_{1})\cap Fr(D_{2})$.
\end{definition}

Now let $X$ and $Y$ be two integral schemes and let $f:X\rightarrow Y$ be a
surjective morphism of finite type. Fixed an algebraic closure $\Omega$ of the function field $k(X)$.

\begin{definition}
A reduced affine covering $\mathcal{C}_{X}$ of $X$ with values in $\Omega $
is said to be \textbf{quasi-galois closed} over $Y$ by $f$ if the
 condition below is satisfied:

There exists a local chart $(U_{\alpha }^{\prime },\phi _{\alpha }^{\prime
};A_{\alpha }^{\prime })\in \mathcal{C}_{X}$ such that $U_{\alpha }^{\prime
}\subseteq \varphi^{-1}(V_{\alpha})$ holds for any $(U_{\alpha },\phi _{\alpha
};A_{\alpha })\in \mathcal{C}_{X}$, for any affine open set $V_{\alpha}$ in $%
Y$ with $U_{\alpha }\subseteq f^{-1}(V_{\alpha})$, and for any
conjugate $A_{\alpha }^{\prime }$ of $A_{\alpha }$ over $B_{\alpha}$, where $%
B_{\alpha}$ is the canonical image of $\mathcal{O}_{ Y}(V_{\alpha})$ in the
function field $k(X)$ via $f$.
\end{definition}

\begin{definition}
An affine covering $\{(U_{\alpha },\phi _{\alpha };A_{\alpha })\}_{\alpha
\in \Delta }$ of $X$ is said to be an \textbf{affine
patching} of $X$ if the map $\phi _{\alpha }$ is the
identity map on $U_{\alpha }=SpecA_{\alpha }$ for each $\alpha \in \Delta .$
\end{definition}

Evidently, an affine patching is reduced.

\begin{theorem}
Let $X$ be {qc}  over $Y$ by $f$ and let the function field $k(Y)$ be contained in $\Omega$.  Then there is a unique
maximal affine covering $\mathcal{C}_{X}$ of $X$ with values in $\Omega $
such that $\mathcal{C}_{X}$ is quasi-galois closed over $Y$ by $f$.
\end{theorem}

\begin{proof}
Trivial.
\end{proof}

It follows that we have the following corollaries.

\begin{corollary}
Let  $\mathcal{C}_{X}$ and $\Omega$ be assumed as in
\emph{Definition 2.1}. Then there are the following statements.
 \begin{itemize}
\item The field $\Omega$ is an algebraic closure of the function field $k(X)$.

 \item The affine covering $\mathcal{C}_{X}$ is a unique maximal affine covering with values in $\Omega$ that is quasi-galois closed over $Y$ by $f$.
 \end{itemize}
\end{corollary}

\begin{corollary}
Any \emph{qc} integral scheme has a unique maximal
affine structure  with values in an algebraic closure of its function field (see \cite{An,An2}).
 \end{corollary}

 Finally, we give a sufficient and essential condition for \emph{qc} schemes.

\begin{condition}
Let $(W,\mathcal{O}_{W})$ be a scheme. For any open set $U$ in $W$, we identify the stalk of $\mathcal{O}_{W}$ at a point $x \in U$
with the stalk of the restriction of $\mathcal{O}_{W}$ to $U$ at the point $x$.
\end{condition}

\begin{theorem}
Let $X,Y$ be integral schemes and let $f :X\rightarrow Y$ be a surjective morphism of finite type.
Suppose that the function field $k(Y)$ is contained in $\Omega$. Then under Condition 3.6, the following statements are equivalent:
\begin{itemize}
\item The scheme $X$ is
qc over $Y$ by $f$.

\item There is a unique maximal affine
patching $\mathcal{C}_{X}$ of $X$ with values in $\Omega $ such that $\mathcal{C}_{X}$ is quasi-galois closed over $Y$ by $f$.
\end{itemize}
\end{theorem}

\begin{proof}
It is immediate from \emph{Lemma 2.13} (in {\cite{An4}}) and \emph{Theorem 3.5} above.
\end{proof}

\end{document}